\documentclass{svmult}

\usepackage[utf8]{inputenc}

\usepackage{amssymb,latexsym,amsmath}
\usepackage[curve]{xypic}

\newcommand\bk{{\Bbbk}}
\newcommand\cO{\operatorname{\mathcal O}}

\newcommand\End{\operatorname{End}}
\newcommand\Enda{\operatorname{End}_A}
\newcommand\Endhb{\operatorname{End}_{hb}}
\newcommand\Nhb{\operatorname{\mathcal N}_{hb}}

\newcommand\Ker{\operatorname{Ker}}

\newcommand\Hom{\operatorname{Hom}}

\newcommand\aff{\operatorname{aff}}

\newcommand\Aut{\operatorname{Aut}}
\newcommand\Authb{\operatorname{Aut}_{hb}}

\begin{document}
\title*{The endomorphisms
  monoid of a homogeneous vector bundle}
\titlerunning{Endomorphisms of homogeneous vector bundles}
\author{L. Brambila-Paz and Alvaro Rittatore}%
\institute{L. Brambila-Paz \at CIMAT, Jalisco S/N
Mineral de Valenciana, 
3624 Guanajuato, Guanajuato, M\'exico\\
\email{lebp@cimat.mx}
\and Alvaro Rittatore \at
 Facultad de Ciencias, Universidad de la Rep\'ublica, Igu\'a 4225, 11400 Montevideo, Uruguay
\email {alvaro@cmat.edu.uy}}


\maketitle

\abstract*{
Let $E$ be a  homogeneous
vector bundle over the abelian variety $A$ and let $\Aut_A(E)$ be the
(algebraic) group of automorphims of $E$ as a vector bundle.   Then
the fiber over $0$ is a  $\Aut_A(E)$-module. We prove that $E$ is
the induced space of this action to the whole group of
automorphims of the homogeneous vector bundle.
The principal significance of this result is that it allows one to
obtain results about the  structure of $E$ and it provides some
insight into the structure of its endomorphisms monoid.
}
\abstract{
Let $E$ be a  homogeneous
vector bundle over the abelian variety $A$ and let $\Aut_A(E)$ be the
(algebraic) group of automorphims of $E$ as a vector bundle.   Then
the fiber over $0$ is a  $\Aut_A(E)$-module. We prove that $E$ is
the induced space of this action to the whole group of
automorphims of the homogeneous vector bundle.
The principal significance of this result is that it allows one to
obtain results about the  structure of $E$ and it provides some
insight into the structure of its endomorphisms monoid.
}

\section{Introduction}
\label{sec:intro}
Let $A$ be an abelian variety over an algebraically closed field of arbitrary
characteristic $\Bbbk$.  A
vector bundle, which will be denoted by $\xi=(E,\rho, A)$ or $\rho:E\to A$,
  is called \emph{homogeneous} if for
all $a\in A$,  $E\cong t^*_a E$, where $t_a:A\to A$ is the
translation by $a$. A \emph{homomorphism} $\lambda:\xi\to \xi '$
between two homogeneous vector bundles $\xi=(E,\rho, A)$ and $\xi
'=(E',\rho ', A)$ is a pair $\lambda:=(f,t_a)$ such that the
following diagram
 \begin{center}
\mbox{ \xymatrix{
E\ar@{->}[r]^f\ar@{->}_\rho[d] &E'\ar@{->}^{\rho'}[d]\\
A\ar@{->}[r]_{t_a }& A } }
\end{center}
commutes and is linear in the fibers. We say that $\lambda$ is an
\emph{isomorphism} if $f:E\to E'$ is an isomorphism of algebraic
varieties; it is clear that in this case the pair $(f^{-1},
t_{-a}):\xi'\to \xi$
is a homomorphism of homogeneous vector bundles.

We denote by $\Hom_{hb}(E,E')$ the set of homomorphisms between
$\xi=(E,\rho, A)$ and $\xi '=(E',\rho ', A)$ and by $\Hom_A(E,E')$
the set of those homomorphisms that fix the base, i.e. $t_a=Id_A.$
If $E=E'$, then $\End_{hb}(E):=\Hom_{hb}(E,E) $ and
$\End_A(E):=\Hom_A(E,E).$ The group of automorphisms, i.e. those
endomorphisms where $f:E\to E$ is an isomorphism, is denoted by $\Authb(E)$.

It is of interest to have a natural and intrinsic characterization of $\End_{hb}(E).$
One of the main reasons is the generalization of a well known result of affine
monoids to normal algebraic monoids. More precisely, an affine monoid $M$ can be
embedded in $\End_\Bbbk(\Bbbk^n)$ as a closed submonoid, for
$n>>0$. In \cite[Theorem 5.3]{kn:ritbr} it was proved that
any normal algebraic monoid $M$
 can
be embedded as a closed submonoid of the endomorphisms monoid
$\End_{hb}(E)$ of an indecomposable homogeneous vector bundle $E$
over the Albanese variety $A(M)$ of $M$.

The aim of this paper is to describe the geometric and algebraic
structure of $\End_{hb}(E)$ and investigate the relation between
these structures and the structure of $E$ as a vector bundle.

In order to state our results we first recall  some known results.
The idea to study the relation of the algebraic structure of $\Aut_{hb}(E)$
 to the structure of $E$ as a vector bundle goes back at least as far as \cite{kn:miy}, where
Miyanishi considers homogeneous bundles. In \cite{kn:Mu78}, Mukai
 describes the category of homogeneous vector bundles over an abelian variety. Brion
and the second author proved in \cite{kn:ritbr}  that $\Endhb(E)$
is an algebraic monoid with unit group
 $\Aut_{hb}(E)$ and they showed that the Albanese morphism $\pi:\End_{hb}(E)\to
 A$ is a
morphism of algebraic monoids, with Kernel $\End_A(E)$. In
particular, the fiber $\End_A(E)$, over the unit element $0\in A$,
is an irreducible affine smooth algebraic monoid, with unit
group $\Aut_A(E):=\pi^{-1}(0)\cap \Aut _{hb}(E)$.

It  follows from \cite{kn:ritbr} that
$\pi:\Endhb(E)\to A$ is a homogeneous vector bundle with fiber
isomorphic to $\End_A(E)$. For any indecomposable homogeneous vector bundle $\rho:E\to A$
 of rank $r$  we prove that $\End_{hb}(E)\to A$ has rank at
 most $1+\frac{r(r-1)}{2}$ (see
Theorem \ref{thm:enhdvb}). Moreover, $\End_{hb}(E)\to A$ is
obtained by successive extensions of a line bundle $L$ associated
with $E$ (see Theorem \ref{thm:Endhbext}). In particular,  if $L$ is
a homogeneous line bundle, then $\End_{hb}(L)\cong L$ (see Lemma
\ref{lem:linemon} and Corollary \ref{coro:endofLisL}).

In order to describe the structure of $\End_{hb}(E)$ as a vector bundle
denote by $E_0$ the fiber over $0\in A$. Recall that \emph{the
induced space} $\Authb(E)*_{\Aut_A(E)}E_0$ is defined as the
geometric quotient of $\Authb(E)\times E_0$ under the diagonal
action of $\Aut_A(E)$ (see Definition \ref{defi:algo} below). In
Theorem \ref{thm:structure} we prove the following statement.

 \emph{ An indecomposable homogeneous vector bundle
$\rho:E\to A$ is the induced space from the action of $\Aut_A(E)$
on $E_0$ to the action of the automorphisms group $\Authb(E)$ on $E$, i.e.
\[
E\cong \Authb(E)*_{\Aut_A(E)}E_0.
\]}

The advantage of using the above description lies on the fact that
it allows us to describe the structure of $\End_{hb}(E)$ also when
$E$ is decomposable. That is, if $E=\bigoplus_{i,j} L_i\otimes
F_{i,j}$, where $L_i$ is a homogeneous line bundle and $F_{i,j}$
an unipotent bundle, then
\[
\End_{hb}(E)\cong  \bigoplus_i L_i\otimes
\bigl(\oplus_{j,k}\Hom_{hb} (F_{i,j},F_{i,k}) \bigr),
\]
and as algebraic monoid $\Endhb(E)$  decomposes as
\[
\Endhb(E)=\Authb(E)\sqcup \Nhb(E),
\]
where $\Nhb(E)$ is the ideal of pseudo-nilpotent elements (see
 theorems \ref{thm:endmon2} and \ref{thm.autnilp}). Moreover,
 the kernel $\Ker\bigl(\End_{hb}(E)\bigr)$ of the
$\End_{hb}(E)$ is the zero
  section
\[
\Ker \bigl(\Endhb(E)\bigr) = \bigl\{ \theta_a:E\to E\mathrel{:}
\theta_a(v_x)=0_{x+a}\
  \forall\, v_x\in E_x \bigr\} .
\]
In particular, $\Ker\bigl(\End_{hb}(E)\bigr)$ is isomorphic to the
abelian variety $A$ (see Proposition \ref{prop:kernelend}).

According to the above results, it is also shown that $\Nhb(E)$ is
also a homogeneous vector bundle, obtained by successive
extensions of $L$ (see Theorem \ref{thm:Endhbext}) and there
exists an exact sequence of vector bundles
 \begin{center}
\mbox{ \xymatrix{ 0 \ar@{->}[r] & \Nhb (E)\ar@{->}[r]&
\Endhb(E)\ar@{->}[r]^-{\rho} & \Endhb(L)\cong L\ar@{->}[r]      &
0 } }
\end{center}
Moreover, the morphisms in the above sequence are compatible with
the structures of semigroup, and if  $E\ncong L$, then the sequence
is non-trivial (see Theorem \ref{thm:sumadirecta}).

The paper is organized as follows: in Section \ref{sec:prelim} we
set up notation and terminology, and review some of the standard
facts on algebraic monoids and homogeneous vector bundles. In
Section \ref{sec:endo} we are concerned with the structure of
$\Hom_{hb}(E,E')$. Our main results
are stated and proved in Section \ref{sec:endmon}.  Section \ref{sec:exam} is devoted to the
study of $\Endhb(E)$ and $\End_A(E)$ when  $E$ is a homogeneous
vector bundle of small rank.

\section{General results}
\label{sec:prelim} In this section we recall  basic results on
algebraic monoids and homogeneous vector bundles over $A$.  For a
deeper discussion of the theory of algebraic monoids we refer the
reader to \cite{kn:Br06,kn:ritbr,kn:Ri98,kn:oam}
 and to \cite{kn:atiyahsh,kn:atiyahconn,kn:atiyahvbec,kn:miy,kn:Mu78} for the theory of
homogeneous vector bundles.

\subsection{Algebraic monoids}
Let $M$ be an {\em algebraic monoid}\/ and $G(M)$ the {\em unit
group}\/ of $M$. It is well known that $G(M)$ is an algebraic
group, open in $M$ (see for example \cite{kn:Ri98}). If $M$ is an irreducible algebraic monoid,
then its \emph{Kernel}, denoted by $\operatorname{Ker}(M)$, is the
minimum closed ideal and always exists. Indeed, if $M$ is an
affine algebraic monoid  its $\Ker(M)$ is the unique closed
$\bigl(G(M)\times G(M)\bigr)$-orbit (see \cite{kn:Ri98,
kn:ritbr}).

Let $M,N$ be algebraic monoids. A morphism of algebraic varieties $\varphi:M\to N$ is a
\emph{morphism of algebraic monoids} if
$\varphi(ab)=\varphi(a)\varphi(b)$ for any $a,b\in M$ and
$\varphi(1_M)=1_N$. If  $\varphi: M\to N$ is an isomorphism of
algebraic monoids we write $M\cong_{am}N$.

\begin{definition}
\label{defi:algo} Let $H\subset G$ be an algebraic subgroup of an
algebraic group $G$ such that $H$ acts on an algebraic variety
$X$. The {\em induced space}\/ $G*_HX$ is defined as the geometric
quotient of $G\times X$ under the $H$-action
$h\cdot(g,x)=(gh^{-1},h\cdot x)$. The class of $(g,x)$ in $G*_HX$
is denoted as $[g,x]$.
\end{definition}

 Under
mild conditions on $X$ (e.g.~$X$ is covered by quasi-projective
$H$-stable open subsets), the induced space always exists.
Clearly, $G*_HX$ is a $G$-variety, for the action induced by
$a\cdot (g,x)=(ag,x)$. The morphism $\pi:G*_HX\to G/H$ induced by
$(g,x)\mapsto [g]=gH$ is a fiber bundle over $G/H$ with fiber
isomorphic to $X$. If moreover $X$ is an $H$-module, then $G*_HX\to
G/H$ is a vector bundle. For more information about induced spaces
we refer the reader to \cite{kn:bb-induced} and \cite{kn:Ti06}.

\begin{remark}\label{remark:strucmonoid}
Chevalley's Structure Theorem for an algebraic group $G$ states
that if  $A(G)$ is the Albanese group of $G$,
then the Albanese morphism $p:G\to A(G)$ fits into an exact
sequence of algebraic groups
\begin{center}
\mbox{ \xymatrix{ 1\ar@{->}[r]&G_{{\aff}}\ar@{->}[r]&
G\ar@{->}[r]^-p
&A(G) \ar@{->}[r] & 0 } }
\end{center}
where $G_{\aff}$ is a normal connected affine algebraic group.

In \cite{kn:Br06,kn:ritbr}, Brion and  Rittatore  generalize
Chevalley's decomposition to irreducible normal algebraic monoids.
 In this case they prove  that if $G$ is the
 unit group,
then $M$ admits a Chevalley's decomposition:

\begin{center}
\mbox{ \xymatrix{
1\ar@{->}[r]&M_{{\aff}}=\overline{G_{{\aff}}}\ar@{->}[r]&
M=\overline{G}\ar@{->}[r]^-p
  &A(G) \ar@{->}[r] &0\\
1\ar@{->}[r]&G_{{\aff}}\ar@{->}[r]\ar@{^(->}[u]&
G\ar@{->}[r]_-{p|_{_G}}\ar@{^(->}[u] &A(G) \ar@{->}[r]\ar@{=}[u] & 0 } }
\end{center}
where $p:M\to A(G)=G/G_{{\aff}}$ (respectively ~$p|_{_G}:G\to A(G)$) is
the Albanese morphism  of $M$ (respectively ~$G$). If $Z^0$
denotes the connected center of $G$, then $A(G)\cong_{am} Z^0/(Z^0\cap
G_{{\aff}})$ and $
M=G\cdot M_{{\aff}}=Z^0\cdot M_{{\aff}}$. Moreover, $M
\cong G*_{G_{{\aff}}}M_{{\aff}}\cong  Z^0*_{Z^0\cap
  G_{{\aff}}}M_{{\aff}}$. Here if $A,B\subset M$, then $A\cdot B=\{
ab \mathrel{:} a\in A, b\in B\}$.
\end{remark}

Let us mention a consequence of this Chevalley's decomposition for
algebraic monoids  that will prove
to be extremely useful in Section \ref{sec:endmon}.

\begin{corollary}
\label{coro:ker} Let $M$ be an irreducible algebraic monoid, with unit group
$G$. Then
$\Ker(M)=G\Ker(M_{{\aff}})G=G\cdot\Ker(M_{{\aff}})=Z^0\cdot\Ker(M_{{\aff}})$
where $Z^0$ is the connected center of $G$.
\end{corollary}

\begin{proof}
Since $M=Z^0\cdot M_{{\aff}}$, it follows that $\Ker(M_{{\aff}})\subset
\Ker(M)$. Hence, $G\cdot\Ker(M_{{\aff}})\cdot G\subset \Ker(M)$. Since
both terms in the last inclusion are $\bigl(G(M)\times
G(M)\bigr)$-orbits, the first equality follows.

It is clear that  $\bigl(G\Ker(M_{{\aff}})G\bigr)\cap
M_{\aff}=\Ker(M_{{\aff}})$ and from the decompositions
$M=Z^0\cdot M_{{\aff}}=M_{{\aff}}\cdot Z^0$ and $G=Z^0\cdot G_{\aff}$,  we deduce
that
\[
G\cdot \Ker(M_{{\aff}})\cdot G=Z^0\cdot \Ker(M_{{\aff}})=G\cdot \Ker(M_{{\aff}}).
\]
\smartqed\qed\end{proof}

\subsection{Homogeneous vector bundles}
\label{label:sect22}

 Recall that a
 vector bundle $\rho: E\to A$ is called \emph{homogeneous} if for any $a\in A$, $E
\cong t_a^* E$, where $t_a$ is the translation by  $a$. A line
bundle $L$ is homogeneous if and only if it is algebraically
equivalent to zero (see \cite[Sect.~9]{Mi86}). In particular, the
trivial bundle $\mathcal{O}_A$ is homogeneous.

Let $\xi=(E,\rho, A)$ and $\xi '=(E',\rho ', A)$ be two
homogeneous vector bundles over $A$. If $(f,t_a) :\xi\to \xi '$ is a  homomorphism, then the morphism
 $t_a:A\to A$ is determined by $f:E\to E'$. Thus, when no confusion can arise,
we will write $(f,t_a)$ simply as $f$.
 If $E,
E'$ are isomorphic as vector bundles we write $E\cong_{vb} E'$. It
is well known that $\Hom_A(E,E')\cong H^0\bigl(A, \Hom(E,E')\bigr)$ and
$\End_A(E)=H^0(A, E^*\otimes E)$. If $E$ is indecomposable, the
\emph{algebra  of endomorphisms} $\End_A(E)$ is a
finite-dimensional $\Bbbk$-algebra and
 $E$ is called \emph{simple} if $\End_A(E)= \Bbbk$.
A homogeneous vector bundle $E\to A$ is indecomposable if and only
if $\End _A(E) = \Bbbk\cdot 1_E \oplus N _A(E)$, where $N
_A(E)\subset \End _A(E)$ is the ideal of all the nilpotent
endomorphisms (see \cite{kn:atiyahconn}). Moreover,
$\Aut_A(E)\cong G_m \times U_A(E)$, where $G_m=\bk^*$ and $U_A(E)$
is the unipotent affine subgroup
$\operatorname{Id}+ N_A(E)$.

\begin{remark}\label{rem:endwithcenter}
\label{thm:endasmonoid} In \cite[Lemma 1.1]{kn:miy} Miyanishi
described the algebraic structure
 of $\Aut_{hb}(E)$. In particular, he proved that, as an algebraic
 group, $\Aut_{hb}(E)$ is an extension of $A$ by
$\Aut_A(E)$, that is, we have the exact sequence of algebraic groups
\begin{center}
\mbox{ \xymatrix{ 1\ar@{->} [r] & \Aut_A(E)\ar@{->}[r]
&\Aut_{hb}(E)\ar@{->}[r]& A \ar@{->}[r]& 0 } }
\end{center}

From the Chevalley's decomposition of
$\End_{hb}(E)$ as an algebraic monoid (see  Remark
\ref{remark:strucmonoid} and \cite{kn:ritbr})
we have that
 $\End(E)_{\aff} = \End_A(E)$  fits in the following exact
sequence of algebraic monoids
\begin{center}
\mbox{ \xymatrix{ 1\ar@{->} [r] & \End_A(E)\ar@{->}[r]
&\End_{hb}(E)\ar@{->}[r]& A \ar@{->}[r]& 0. } }
\end{center}
Moreover, if
$\operatorname{Z}^0_{hb}(E)$ is the connected center of
$\End_{hb}(E)$ and
$\operatorname{Z}^0_A(E)=\operatorname{Z}^0_{hb}(E)\cap \End_A(E)$,
then we have the following isomorphisms of algebraic monoids
\[
\begin{split}
\Endhb(E) & =  \Authb(E)\cdot\End_A(E)
=\operatorname{Z}_{hb}^0(E)\cdot\End_A(E)
\\
& \cong_{am} \Authb(E)*_{\Aut_A(E)}\End_A(E)
 \cong_{am}
\operatorname{Z}_{hb}^0(E)*_{\operatorname{Z}^0_A(E)}\End_A(E).
\end{split}
\]
\end{remark}

\begin{remark} The canonical morphism $\pi:\Hom_{hb}(E,E')\to A$,
$\pi(f,t_a)=a$, is  a fibration over $A$ with fiber  $\Hom_A(E,
t^*_aE')$ for any
 $a \in A.$ Indeed,  given
$(f,t_a)\in \Hom_{hb}(E,E')$ there is a homomorphism $f^a:E\to
t_a^*E'$ of vector bundles over $A$, such that the following
diagram
\begin{center}
\mbox{ \xymatrix{
E\ar@{->}[dr]^-{f^a}\ar@/^0.7pc/[drr]^-{f}\ar@/_0.7pc/[ddr] & &\\
& t_a^*E'\ar@{->}[d]\ar@{->}[r] & E'\ar@{->}[d]\\
 & A\ar@{->}[r]_{t_a}& A
 } }
\end{center}
commutes. That is, for any $a\in A$ there is a natural bijection
between $\pi ^{-1}(a)$ and $\Hom_A(E, t^*_aE')$.
 In particular,
$\Hom_A(E,E')=\pi^{-1}(0)$.
\end{remark}

\begin{definition}
\label{defi:sucext}
We say that a vector bundle $E$ of rank $r>1$ is obtained by \emph{successive extensions
of a vector bundle} $R$, of length $s$, if  there exists a filtration
$$R=E_0\subset E_1\subset E_2\subset E_3\subset \dots \subset E_{s-1} \subset
E_s=E$$ such that $E_i/E_{i-1}\cong R$ for $i=1,\dots , s$.
In other words, there exist extensions
\[
\begin{array} {cccrcl}
\rho _1:  & 0&\longrightarrow  E_0\cong R & \stackrel{i_1}{\longrightarrow}
& E_1&\stackrel{p_1}{\longrightarrow} R \longrightarrow 0
\\
\rho _2: & 0&\longrightarrow  E_1&\stackrel{i_2}{\longrightarrow}&  E_2
&\stackrel{p_2}{\longrightarrow} R \longrightarrow 0\\
\vdots & && \vdots&\\
\rho _s: & 0&\longrightarrow  E_{s-1}&\stackrel{i_{s}}{\longrightarrow}&  E_s\cong E
&\stackrel{p_{s}}{\longrightarrow} R \longrightarrow 0
\end{array}
\]
If $R$ is the trivial bundle
$\cO_A$, then $E$ is called \emph{unipotent}.
We call $(\rho _1,\dots,\rho _{s})$ the \emph{extensions associated
with} $E$. Note that  $gr(E)=\oplus E_i/E_{i-1}$, the
graded bundle  associated with this
filtration, is isomorphic to $\oplus^s R$. In particular, $E$ is
unipotent if  $gr(E)=\oplus^s \mathcal{O}_A$.
\end{definition}

\begin{proposition}
\label{propdimend} If $E$ is a vector bundle of rank $r$ obtained
by successive extensions of a vector bundle  $R$, of length $s$,
then $2\leq \dim_\bk \End_A(E)\leq 1+ r(r-1)/2$.
\end{proposition}

 \begin{proof}
Let
\[
\begin{array} {crcl}
\rho _1:  & 0\longrightarrow  R  \stackrel{i_1}{\longrightarrow} &
E_1&\stackrel{p_1}{\longrightarrow}
R \longrightarrow 0
\\
\vdots & & \vdots&\\
\rho _s: & 0\longrightarrow
E_{s-1}\stackrel{i_{s}}{\longrightarrow}&  E_s\cong E
&\stackrel{p_{s}}{\longrightarrow} R \longrightarrow 0
\end{array}
\]
 be the extensions associated with $E$. The
composition $\varphi= i_{s}\circ\dots \circ i_2\circ i_1 \circ
p_{s}\neq 0$
 defines a non invertible endomorphism of $E$. Therefore,
 $2\leq \dim_\bk\End_A(E)$.

As in \cite[Prop. 1.1.9]{kn:lebp1} we have that  $\dim
\End_A(E)\leq 1+ r(r-1)/2$. Note that $1+ r(r-1)/2$ is the dimension of
the upper triangular matrices in $\End_\bk(E_a)$.
  Indeed, the fiber $E_a$ has a flag
invariant under $e_a(\End_A(E))$ where $e_a:\End_A(E)\to
\End_\bk(E_a)$, $e_a(f)=f|_{_{E_a}}$ is the restriction to the
fiber $a\in A$. Hence,
\[
\dim \End_A(E)\leq 1+ r(r-1)/2.
\]
 \smartqed\qed\end{proof}

\begin{remark} Let  $\rho:E\to A$ be a vector bundle. From \cite{kn:miy,kn:Mu78}  we have:
\label{thm:miyaandmukai}
\begin{enumerate}
\item  If $E$ is an unipotent vector bundle, then $E$ is  homogeneous.

 \item
 $E$ is an indecomposable homogeneous vector bundle if and only
if $E$  is obtained by successive
extensions of a homogeneous line bundle $L$.   Moreover, one can choose the
associated filtration  $0\subsetneq E_1=L\subsetneq \dots
\subsetneq E_i\subsetneq \dots \subsetneq E_n=E$ in such a way
that $f(E_i)\subset E_i$ for all $f\in\Authb(E)$, i.e. the filtration is $\Authb(E)$-stable.

Recall that in this case $E\cong L\otimes F$,
where $L\in \operatorname{Pic}^0(A)$ and
$F$ is an indecomposable unipotent vector bundle.

 \item  $E$ is homogeneous if and only if  $E$
decomposes as a direct sum $E=\bigoplus L_i\otimes F_i$, where
$L_i\in \operatorname{Pic}^0(A)$ and $F_i$ is a unipotent vector
bundle.

\end{enumerate}
\end{remark}

\begin{theorem}
\label{thm:enhdvb} Let $E\to A$ be a homogeneous vector bundle of
rank $r$. Then $\pi: \End_{hb}(E)\to A$ is a homogeneous vector bundle
with fiber isomorphic to $\End_A(E)$. Moreover, if $E$ is
 indecomposable, then
 $\operatorname{rk}(\Endhb(E))\leq  1+ r(r-1)/2.$
\end{theorem}

\begin{proof}
Recall that $\Endhb(E)\cong_{am} \Authb(E)*_{\Aut_A(E)}\End_A(E)$
(see Remark \ref{rem:endwithcenter}). From general properties of
the induced action and the fact that $\End_A(E)$ is a finite
dimensional algebra,  it follows that $\Endhb(E)\to A$ is a vector
bundle (see for example \cite{kn:Ti06}). Moreover, since $\Authb(E)$
acts by left multiplication on
$\Endhb(E)$ we have that $\Endhb(E)$ is homogeneous. Indeed, given
$a\in A$, there exists $(f, t_a)\in  \Authb(E)$ and, if
$\ell_f:\End_{hb}(E)\to \End_{hb}(E)$ denotes the isomorphism
$\ell_f(h)=f\circ h$ for $h\in\End_A(E)$, then
$\alpha(\ell_f)=t_a$.

The second part follows from Proposition \ref{propdimend}.
\smartqed\qed\end{proof}

\begin{remark} In \cite[Proposition 6.13]{kn:Mu78} Mukai  proved that
  homogeneous vector bundles are Gieseker-semistable (see
  \cite{kn:gieseker}). Moreover, a homogeneous vector
bundle $E$ is Gieseker-stable if and only if it is simple. It
follows from Proposition \ref{propdimend}  that a homogeneous
vector bundle $E$ is Gieseker-stable if and only if it is a
homogeneous line bundle.
\end{remark}

\section{The endomorphisms monoid of a homogeneous vector bundle}
\label{sec:endo}

The affine algebraic  group  $\Aut_A(E)$ acts in at least two
different ways on $\End_A(E)$, either by post-composing, $f\cdot
h=f\circ h$,  or by pre-composing, $f\cdot h= h\circ f^{-1}$, with
$f\in \Aut_A(E)$ and  $h\in \End_A(E)$. This allow us to endow
$\Endhb(E)$ with two structures of homogeneous vector bundle.
However, since
$$\Authb(E)*_{\Aut_A(E)}\End_A(E)\cong_{vb}
\operatorname{Z}_{hb}^0(E)*_{\operatorname{Z}_{A}^0(E)}\End_A(E),$$
one can prove that, in fact, these structures are the same.
Instead of proving this in full details, we give  slightly more
general results relating the structures of vector bundle of
$\Hom_{hb}(E,E')$ (see Theorems \ref{thm:endmon01} and \ref{thm:endmon02}).

\begin{proposition}\label{lemma:hombundles1}
Let $E$ and $E'$ be two  vector bundles over $A$. Suppose $E'$ is
homogeneous. The inclusion
$\operatorname{Z}_{hb}^0(E')\hookrightarrow \Authb(E')$ induces an
isomorphism of the homogeneous vector bundles
\[
\operatorname{Z}_{hb}^0(E')*_{\operatorname{Z}_{A}^0(E')}\Hom_A(E,E')\cong_{vb}
\Authb(E')*_{\Aut_A(E')}\Hom_A(E,E'),
\]
where   $\operatorname{Z}_{A}^0(E')$ and $\Aut_A(E')$ act on
$\Hom_A(E,E')$ by post-composing.
\end{proposition}

\begin{proof}

Recall that the induced space
$P=\Authb(E')*_{\Aut_A(E')}\Hom_A(E,E')$ is a vector bundle over
$\Authb(E')/\Aut_A(E')=A$ and that a vector bundle $R$ is homogeneous if the restriction map
$Aut(R)\to A$ is surjective. It is clear that the canonical action of $\Authb(E')$ over $P$
induces a morphism of  groups $\varphi: \Authb(E')\to \Authb(P)$,
$\varphi( f,t_a)=(\widetilde{f}, t_a),$ where
$\widetilde{f}\bigl(\bigl[(h,t_b), h'\bigr]\bigr)= \bigl[(f\circ
h,t_{b+a}), h'\bigr]$. Since the  projection
$\Authb(E')\to A$ defined as $(f,t_a)\mapsto a$, is surjective, it follows
that the canonical projection $\Authb(P)\to A $ is also
surjective. In other words, the vector bundle $P$ is homogeneous.
The same conclusion can be drawn for
$Q=\operatorname{Z}_{hb}^0(E')*_{\operatorname{Z}_{A}^0(E')}\Hom_A(E,E')$.
The inclusion $\operatorname{Z}_{hb}^0(E')\hookrightarrow
\Authb(E')$ induces a morphism of homogeneous vector bundles $Q\to
P$ which is bijective, this clearly forces $Q\cong_{vb}P$.
\smartqed\qed\end{proof}

\begin{theorem}
\label{thm:endmon01} Under the same hypothesis of Proposition
\ref{lemma:hombundles1}, consider the projection
$\pi:\Hom_{hb}(E,E')\to A$, $\pi (f,t_a)=a$, and let
\[
\begin{array}{l}
\pi':\Authb(E')*_{\Aut_A(E')}\Hom_A(E,E')\to A\\
  \pi'(g, h)=[g]\in
\Authb(E')/\Aut_A(E')\cong A
\end{array}
\]
 be the canonical
projection. Then
 there exists a bijection
 \[ \phi : \Authb(E')*_{\Aut_A(E')}\Hom_A(E,E') \to \Hom _{hb}(E,E')
\]
 such that the following diagram
\begin{equation}
\label{eqn:diag2}\mbox{
\xymatrix{
\Authb(E')*_{\Aut_A(E')}\Hom_A(E,E')\ar@{->} [rr]^-{\phi}\ar@{->}[rd]_{\pi'} & &
\Hom _{hb}(E,E')\ar@{->}[dl]^{\pi}\\
& A&
}
}
\end{equation}
is commutative and $\phi$ is linear in the fibers. Moreover,
  $\phi$ induces the structure of a homogeneous vector
bundle on $\Hom_{hb}(E,E')$
 and
$$\Hom _{hb}(E,E')
\cong_{vb}\operatorname{Z}_{hb}^0(E')*_{\operatorname{Z}_{A}^0(E')}\Hom_A(E,E').$$
\end{theorem}

\begin{proof} From Proposition \ref{lemma:hombundles1} it is
  sufficient to prove the existence of $\phi$. Note that given $(g,t_a) \in
  \Authb(E')$ and $h\in \Hom_A(E,E')$
the following  diagram is commutative:

\begin{center}
\mbox{ \xymatrix{ E\ar@{->} [rr]^-{h}\ar@{->}[dr] & & E'\cong
t^*_aE'\ar@{->} [r]^-{g}\ar@{->}[dl]  &
E'\ar@{->}[d]\\
& A\ar@{->}[rr]_{t_a} & &  A
}
}
\end{center}

If $\varphi :\Authb(E')\times \Hom_A(E,E')\to \Hom _{hb}(E,E')$ is
given by $\varphi\bigl((g,t_a),h\bigr)=(g\circ h,t_a)$, then
$\varphi$ is constant on the $\Aut_A(E')$-orbits, and hence
induces a homomorphism  $$\phi:
\Authb(E')*_{\Aut_A(E')}\Hom_A(E,E')\to \Hom(E,E').$$ By
construction, $\phi$ makes the diagram \eqref{eqn:diag2}
commutative.

To prove the surjectivity of $\phi$ let  $(f,t_a)\in \Hom(E,E')$.
 Since $E'$ is homogeneous,
there exists $(g,t_{-a})\in \Authb(E')$ with $a\in A$ such that the
following  diagram is commutative:
\begin{center}
\mbox{ \xymatrix{
E\ar@{->}[r]^f\ar@{->}[d] & E'\cong t^*_aE'\ar@{->}[r]^g\ar@{->}[d] &E'\ar@{->}[d]\\
A\ar@{->}[r]^{t_{a}}
\ar@/_1pc/[rr]_{t_0=\operatorname{Id}}&A\ar@{->}[r]^{t_{-a}}& A }
}
\end{center}

Hence, the composition $(g,t_{-a})\circ (f,t_a)=(g\circ f, t_0)$
defines a homomorphism $g\circ f:E\to E'$ of vector bundles over
$A$. Moreover,
\[
\phi\bigl(\bigl[(g^{-1},t_a), g\circ f\bigr]\bigl)= (g^{-1}\circ g\circ
f, t_a)= (f,t_a),
\]
and hence $\phi$ is surjective.

We claim that $\phi$ is injective. Indeed, if
\[
\bigl[(g_1,t_{a_1}),h_1\bigr], \bigl[(g_2,t_{a_2}),h_2\bigr]\in
\Authb(E)*_{\Aut_A(E)}\Hom_A(E,E')
\]
are such that $ \phi\bigl(\bigl[(g_1,t_{a_1}),h_1)\bigr]\bigr)
=\phi\bigl(\bigl[(g_2,t_{a_2}),h_2\bigr] \bigr)$, then, by
definition of $\phi $, we have that $g_1\circ h_1=g_2\circ h_2$
and  $t_{a_1}=t_{a_2}$. Therefore, $a_1=a_2$ and hence
$g_2^{-1}\circ g_1\in \Aut_A(E)$. This completes the proof.
\smartqed\qed\end{proof}

Theorem \ref{thm:endmon01} states that if $E'$ is homogeneous,
then $\Hom_{hb}(E,E')$ is homogeneous.
Under the assumptions of Proposition \ref{lemma:hombundles1}  with ``$E'$ homogeneous'' replaced
with ``$E$ is homogeneous'' and ``post-composing'' by ``pre-composing
with the inverse''
we obtain an analogue of Theorem
\ref{thm:endmon01} which may be proved in much the same way.

\begin{theorem}
\label{thm:endmon02}
 Let $E,E'$ be vector bundles over $A$. If $E$ is homogeneous,
 then there exists a bijection
 \[
\psi : \Authb(E)*_{\Aut_A(E)}\Hom_A(E,E') \longrightarrow
\Hom(E,E')
\]
 such that the following diagram is
commutative
\begin{center}
\mbox{
\xymatrix{
\Authb(E)*_{\Aut_A(E)}\Hom_A(E,E')\ar@{->} [r]^-{\psi}\ar@{->}[d]_{\pi'}  &
\Hom(E,E')\ar@{->}[d]^{\pi}\\
 A\ar@{->} [r]_-{-\operatorname{Id}}& A
}
}
\end{center}
and $\psi$ is linear when restricted to the fibers. Moreover,
$\psi$ induces a  structure of homogeneous vector bundle on
 $\Hom _{hb}(E,E')$ and
$$\Hom
_{hb}(E,E') \cong_{vb}
\operatorname{Z}_{hb}^0(E)
*_{\operatorname{Z}_{A}^0(E)}\Hom_A(E,E').$$
\end{theorem}

 If $E,E'$ are both homogeneous vector bundles, it
is not clear \emph{a priori} that the two  vector bundle structures
on $\Hom_{hb}(E,E')$ given in Theorem \ref{thm:endmon01} and
Theorem \ref{thm:endmon02} are the same. In order to prove that the
structures coincide, we first deal with the case of $\End_{hb}(E)$,
for $E$ a homogeneous vector bundle.

Consider  the vector bundle
 $P=\operatorname{Z}_{hb}^0(E)*_{\operatorname{Z}_{A}^0(E)}\End_A(E)$,
 where $\operatorname{Z}_{A}^0(E)$ acts on $\End_A(E)$ by
 post-composing and let  $Q$ be the vector bundle $
Q=\operatorname{Z}_{hb}^0(E)*_{\operatorname{Z}_{A}^0(E)}\End_A(E)$,
where $\operatorname{Z}_{A}^0(E)$ acts on $\End_A(E)$  by
pre-composing with the inverse.
An easy calculation shows that the morphism $\xi:
\operatorname{Z}_{hb}^0(E)\times \End_A(E)\to Q$ given by
$\xi\bigl( (f,t_a),h\bigr)=\bigr[(f^{-1},t_{-a}), h\bigr]$, with $ h\in \End_A(E)$, induces
an isomorphism of vector bundles $\widetilde{\xi}: P \to
(\operatorname{-Id})^* Q$ and hence we have the following corollary.

\begin{corollary}\label{cor:endbundlesim1}
If $E$ is a homogeneous vector bundle, then the
structures of homogeneous vector bundle defined on $\Endhb(E)$ by
$\phi$ in Theorem \ref{thm:endmon01} and $\psi$ in Theorem
\ref{thm:endmon02} are isomorphic. 
\end{corollary}

\begin{remark}\label{lem:endbundlesim1} Suppose now that  $E=\bigoplus_{i} E_i$ and $E'=\bigoplus_j
E_j'$ are two homogeneous vector bundles.
Consider  $\Hom_{hb}(E,E')$ and $\Hom_{hb}(E_i,E'_j)$, along with their structure as
homogeneous vector bundles coming from either Theorem
\ref{thm:endmon01} or Theorem  \ref{thm:endmon02}. Then
\[
\Hom_{hb}(E,E')\cong_{vb}  \bigoplus_{i,j} \Hom_{hb} (E_i,E_j').
\]
In particular,
\[
\Endhb(E)\cong_{vb}  \bigoplus_{i,j}\Hom_{hb}(E_i,E_j).
\]
Indeed, the canonical inclusions $\varphi_{ij}:
\Hom_{hb}(E_i,E'_j)\hookrightarrow \Hom_{hb}(E,E')$ are  morphisms of
vector bundles and
 induce an isomorphism
 \[ \varphi:\bigoplus_{i,j} \Hom_{hb} (E_i,E_j')\to
 \Hom_{hb}(E,E').
 \]
\end{remark}

\begin{theorem}
\label{thm:endmon1} Let $E$ and $E'$ be homogeneous vector
bundles. The structures of vector
bundle on $\Hom_{hb}(E,E')$ given in Theorem \ref{thm:endmon01}
and  in Theorem \ref{thm:endmon02} are isomorphic.
\end{theorem}

\begin{proof}

Consider the vector bundle $E\oplus E'$. From Remark \ref{lem:endbundlesim1} we have that
\[
\End_{hb}(E\oplus E') \cong_{vb} \Hom_{hb}(E,E') \oplus
\Hom_{hb}(E',E) \oplus \End_{hb}(E)\oplus \Endhb(E').
\]
Applying  Corollary \ref{cor:endbundlesim1} we obtain an isomorphism between the structures of
$\Hom_{hb}(E,E')$  and $\Hom_{hb}(E',E)$ given in Theorem
\ref{thm:endmon01} and \ref{thm:endmon02}.
\smartqed\qed\end{proof}

\section{Relationship between the structure of a homogeneous bundle
and its endomorphisms monoid} \label{sec:endmon}

\subsection{The  homogeneous vector bundle as an induced space}

We begin this section by showing that a homogeneous vector bundle
$E\to A$ is obtained  as an extension of the fiber $E_0$ over $0$
by the principal bundle $\Authb(E)\to A$.

\begin{theorem}
\label{thm:structure} Let $E\to A$ be a homogeneous  vector
bundle. Then, as vector bundles over $A$,
\[
E\cong _{vb}\Authb (E) *_{\Aut_A(E)}E_0\cong
_{vb}\operatorname{Z}^0_{hb}(E) *_{\operatorname{Z}^0_A(E)}E_0.
\]
\end{theorem}

\begin{proof}
Recall that $A\cong \Authb(E)/\Aut_A(E)$. Define
$$\phi:\Authb (E)\times  E_0\to E$$ as
$\bigl((f,t_a),v\bigr)\mapsto f(v)\in E_a$, where $(f,t_a)\in
\Authb (E)$ and $v\in E_0$.   Clearly,   $\phi$ is constant on the
$\Aut_A(E)$-orbits, and hence induces a  homomorphism
$\varphi:\Authb (E) *_{\Aut_A(E)}E_0\to E$. In fact $\phi$ is an isomorphism of
vector bundles. Indeed, given $a\in A$, consider  $(f,t_a)\in
\Authb(E)$. Then
\[
\bigl(\Authb (E) *_{\Aut_A(E)}E_0\bigr)_a=\bigl\{ \bigl(
(f,t_a),v\bigr)\mathrel{:} v\in E_0\bigr\},
\]
and $f|_{E_0}:E_0\to E_a$ is a linear isomorphism.  It follows that the
restriction $\varphi_a: \bigl(\Authb (E)
*_{\Aut_A(E)}E_0\bigr)_a\to E_A$, $\varphi_a\bigl((f,t_a),
v\bigr)=f(v)$, is a linear isomorphism and $\varphi$ is an
isomorphism of vector bundles as claimed.

Since $\operatorname{Z}^0_{hb}(E)\to A$ is surjective, it is clear
that we can apply the same argument to prove that $E\cong
\operatorname{Z}^0_{hb}(E) *_{\operatorname{Z}^0_A(E)}E_0$.
\smartqed\qed\end{proof}

Theorem \ref{thm:structure} allows us to provide some insight into the structure of
homogeneous vector bundles.

\begin{corollary}
\label{coro:indesc} Let $E\to A$ be an indecomposable  homogeneous
vector bundle. Then:
\begin{itemize}
\item[(i)]\   $E_0$ is an indecomposable
$\operatorname{Z}^0_A(E)$-module.

\item[(ii)]\    $E_0$ is an indecomposable $\End_A(E)$-module;

\end{itemize}

\end{corollary}

\begin{proof}
We prove (i) by way of contradiction. So assume that
 $E_0\cong
V_1\oplus V_2$ as  $\operatorname{Z}^0_A(E)$-modules. Then, we
have the isomorphisms
\[
\bigl(\operatorname{Z}^0_{hb}(E)*_{\operatorname{Z}^0_A(E)}
V_1\bigr)\oplus
\bigl(\operatorname{Z}^0_{hb}(E)*_{\operatorname{Z}^0_A(E)}
V_2\bigr) \cong_{vb}
\bigl(\operatorname{Z}^0_{hb}(E)*_{\operatorname{Z}^0_A(E)}
(V_1\oplus V_2)\bigr)\cong_{vb} E
\]
 as vector bundles over $A$, where the last isomorphism is given by
 Theorem \ref{thm:structure}. It follows that $E$ is decomposable, a contradiction.
 It is clear that (i) implies (ii).
\smartqed\qed\end{proof}

The converse of Corollary \ref{coro:indesc} is false in general, as the
following example shows.

\begin{example}
Let $E=L\oplus L$ where $L$ is a homogeneous line bundle over $A$.
Then $E$ is a decomposable homogeneous vector bundle. However, $E_0$ is an
indecomposable $\Aut_A(E)$-module, since $\Aut_A(E)\cong
\operatorname{GL}_2(\bk)$.
\end{example}

 Denote by $I_r=\bigoplus_{i=1}^r \mathcal O_X=X\times \bk^r$,
the trivial homogeneous vector
 bundle over $X$ of rank $r$, where $X$ is a complete
homogeneous space. In \cite[Lemma 1.4]{kn:miy}, Miyanishi
 gives a characterization of
$I_r\to X$  in terms of the existence of schematic sections for
certain fibrations.  Theorem \ref{thm:structure} allows us to
characterize $I_r$ over an abelian variety in a simpler way in terms of their
endomorphisms monoid.

 Recall that a schematic
section of a fibration $\pi: \Authb(E)\to A$ is a morphism
$\sigma:A\to \Authb(E)$ such that
$\pi\circ\sigma=\operatorname{Id}_A$.

\begin{corollary}
\label{coro:casotrivial} \label{coro:schematic}
 Let $E\to A$ be a homogeneous vector bundle of rank $r$. Then the
 following assertions are equivalent:
\begin{enumerate}
\item $E\cong_{vb}I_r$.

\item $\Endhb(E)\cong_{am} A\times \operatorname {End}(\bk^n)$.

\item  $\pi: \Authb(E)\to A$ has a schematic section.
\end{enumerate}
\end{corollary}

\begin{proof}
It is clear that the endomorphisms monoid $\Endhb(I_r) $ of $I_r$,
satisfies $\Endhb(I_r)\cong_{am} A\times \operatorname
{End}(\bk^n)$, so (1) implies (2) and (3).

$ (2) \Rightarrow  (1)$ If
$\End_{hb}(E)\cong_{am} A\times \operatorname {End}(\bk^n)$, then
\[
E\cong_{vb}  \bigr(A\times \operatorname{GL}_n(\bk)\bigr)
*_{\operatorname {GL}_n(\bk)} \bk^n\cong_{vb} A\times \bk^n,
\]
since $E\cong_{vb} \Authb(E)*_{\Aut_A(E)} E_0$.

$ (3) \Rightarrow  (1)$. Let now $\sigma :A\to
\Authb(E)$, $\sigma(a)=\bigl(\sigma_1(a),t_a\bigr)$, be a
schematic section, and let  $\varphi:A\times E_0\to
\Authb(E)*_{\Aut_A(E)}E_0\cong_{vb}E$ be the morphism  given  by
$\varphi (a,v)= \bigl[(\sigma_1(a),t_a), v\bigr]$. Clearly,
$\varphi$ is a homomorphism of homogeneous vector bundles. It is
enough to prove that $\varphi$ is injective, since
both vector bundles have the same rank.

Suppose that $(a,v),(a',v')\in A\times E_0$ are such that
$\varphi(a,v)=\varphi(a,v')$.  Then
\[
 \bigl[(\sigma_1(a),t_a),
v\bigr]=  \bigl[(\sigma_1(a'),t_{a'}), v'\bigr],
\]
and it follows that $a=a'$. Therefore $v=v'$ and hence $\varphi$  is an isomorphism.
 \smartqed\qed\end{proof}

\begin{corollary}
\label{coro:incompleto} Let $E,E'$ be two homogeneous vector
bundle over $A$. Then the following statements are equivalent:

\noindent {\rm (i)} $E\cong_{vb} E'$;

\noindent {\rm (ii)}  $\Authb (E)\cong_{am} \Authb(E')$, and
$E_0\cong E'_0$ as rational $ \Aut_A(E)$-modules;

\noindent {\rm (iii)}  $\Authb (E)\cong_{am} \Authb(E')$, and
$E_0\cong E'_0$ as rational $
\operatorname{Z}^0_A(E)$-modules.
\end{corollary}

\begin{proof}
The implications (i) $\Longrightarrow$ (ii) $\Longrightarrow$
(iii) are clear.

Assume that (iii) holds. Let $\psi: \Authb(E)\to \Authb(E')$ be an
isomorphism of algebraic groups and let $\Phi :E_0\to E'_0$ be a
morphism of $\operatorname{Z}^0_A(E)$-modules. Then the morphism $\varphi:
\operatorname{Z}^0(E)\times E_0\to E'$, defined as
$\varphi(g,v)=\psi(g)\bigl(\Phi (v)\bigr)$ induces the required
isomorphism $E\to E'$.
\smartqed\qed\end{proof}

\begin{remark} It is well known that there exist homogeneous vector
  bundles $E, E'$ such that $E\not\cong E'$ whereas
  $\Aut_A(E)\cong_{am} \Aut _A(E')$.  Even the stronger condition $\End _A(E)\cong_{am} \End
  _A(E')$ is not sufficient in order to guarantee that $E\cong E'$.
However,  Corollary \ref{coro:casotrivial} shows that the
trivial bundle is characterized by its endomorphisms monoid. One
can see that $\operatorname{Z}^0_A(I_r)=\bk^*\operatorname{Id}$
acts by homotheties in the fiber. In general, the group
$\operatorname{Z}_A^0(E)$ could be larger and there  could exist
two different irreducible representations of the same dimension.
This is the main problem we encounter when trying to generalize
 Corollary
\ref{coro:casotrivial}. Thus, it raises the following question.
\end{remark}

\begin{question}
\label{question} Let $E,E'\to A$ be two indecomposable homogeneous
vector bundles over $A$. Does the existence of an isomorphism
  $\Authb(E)\cong_{am} \Authb(E')$ (or
$\Endhb(E)\cong_{am} \Endhb(E')$) imply that $E\cong E'$?
\end{question}

The following lemma is a straightforward generalization of
\cite[Lemma 4.3]{kn:KKLV} and may be proved in much the same way
(see also \cite[Theorem 2]{kn:oam}).

\begin{lemma}
\label{lem:linemon} Let $\rho:L\to A$ be a homogeneous line
bundle. Then there exists a structure of a commutative algebraic
monoid on $L$ such that $\rho$ is a morphism of algebraic monoids.
The fiber $\rho^{-1}(0)= L_0\cong \bk$ is central in $L$.
Moreover, the unit group is $G(L)=L\setminus \Theta(L)$, where
$\Theta(L)$ is the image of the zero section of $L$.
\end{lemma}

\begin{corollary}
\label{coro:endofLisL} If $L\to A$ is a homogeneous line bundle,
then $\Endhb(L)\cong_{vb} L$.
\end{corollary}

\begin{proof}
By Lemma \ref{lem:linemon} $L$ is an algebraic monoid. For any
$x\in L$ let $l_x :L\to L$ be the endomorphisms defined as
$l_x(y)=xy$ (the product on the algebraic monoid $L$). Hence, $L$
is a sub-bundle of $\Endhb(L)$. But $\Enda(L)\cong \bk$; hence
$\Endhb(L)$ is a line bundle, and $L=\Endhb(L)$.
\smartqed\qed\end{proof}

\subsection{The  vector bundle structure of $\End_{hb}(E)$}

For $i=1,2$ let $E_i$ be a homogeneous vector bundle over $A $. In
order to study the structure of $\Hom_{hb}(E_1,E_2)$ as a vector
bundle it suffices to assume that $E_i$ is indecomposable (see
Remark \ref{lem:endbundlesim1}). In this case, we have that
$E_i\cong L_i\otimes F_i$, where $L_i\in \operatorname{Pic}^0(A)$
and $F_i$ is a unipotent homogeneous vector bundle.

\begin{proposition}
\label{prop:lotimesf}  Let $E_i\cong L_i\otimes F_i$ be
indecomposable homogeneous vector bundles of rank
$r_i=\operatorname{rk}(E_i)$, for $i=1,2$.

\noindent {\rm (1)} If $L_1\not\cong_{hb}L_2$, then
 \[
\Hom_{hb}(E_1,E_2) = \{ \theta_a:E_1\to E_2 \mathrel{:} a\in A\}
\cong_{vb} A\times\{0\},
\]
where if $v\in (E_1)_x$, then $\theta_a(v)= 0_{a+x}\in E_2$.

\noindent {\rm(2)} If $L_1\cong_{vb}L_2$, then
\[
\Hom_{hb}(E_1,E_2)\cong_{vb} L_1\otimes \Hom_{hb}(F_1,F_2).
\]
\medskip

Moreover, $\Endhb(E_1)\cong_{vb} L_1\otimes \Endhb(F_1)$.
\end{proposition}

\begin{proof}
By Proposition \ref{lemma:hombundles1},
\[
\Hom_{hb}(E_1,E_2) \cong
\Authb(E_1)*_{\Aut_A(E_1)}\Hom_A(E_1,E_2).
\]
 We claim that if $L_1\not\cong L_2$,
 then $\Hom_A(E_1,E_2)= 0$, and hence,
\[
\Hom_{hb}(E_1,E_2)\cong \Authb(E_1)*_{\Aut_A(E_1)}\{0\}\cong
A\times \{0\}.
\]

Indeed, suppose that there exists $0\not=\varphi\in \Hom_A(E_1,E_2)$ and let
$$L=H_0\subset H_1\subset H_2\subset \dots \subset H_{r_1-1} \subset
H_{{r_1}}=E_1$$ be the filtration associated with $E_1$. Let $k\in \{0,\dots, r_1-1\}$
be such that $H_k\subset \Ker(\varphi)$ but $ H_{k+1} \not\subset
\Ker(\varphi)$. Let   $j\in \{0,\dots, r_2-1\}$ be such that
$\operatorname{Im}(\varphi)\subset K'_{j+1}$,
$\operatorname{Im}(\varphi)\not\subset K'_{j}$ where
$$0=K_0\subset K_1\subset K_2\subset \dots \subset K_{r_2-1} \subset
K_{{r_2}}=E_2$$ is the filtration associated with $E_2$.

 Then
$\varphi$ induces a non zero morphism $\widetilde{\varphi}:
L_1\cong H_{k+1}/H_k\to K_{j+1}/K_j\cong L_2.$  Since both are
algebraically equivalent to zero, $\widetilde{\varphi}$ is an
isomorphism, and  $L_1\cong L_2$ as claimed.

Suppose now that $L_1\cong_{vb}L_2= L$. Then
\[
\Hom_A(E_1,E_2) \cong (L\otimes F_1)^\vee\otimes (L\otimes F_2)
\cong F_1^\vee\otimes F_2\cong \Hom_A(F_1,F_2).
\]

It follows that $\Hom_{hb}(E_1,E_2)$ and $L\otimes
\Hom_{hb}(F_1,F_2)$ are homogeneous vector bundles of the same
rank. Consider now the homomorphism of vector bundles $\varphi:
L\otimes \Hom_{hb}(F_1,F_2)\to \Hom_{hb}(E_1,E_2)$ given by $
\varphi\bigl((l,t_a)\otimes (h,t_a)\bigr) = (l\otimes h, t_a)$,
where we use the identification $L\cong_{vb} \End_{hb}(L)$, and $(l\otimes
h)( v\otimes w)= l(v)\otimes h(w)$, for $v\otimes  w\in L\otimes
F$. Since $\varphi$ is
 injective, it is an isomorphism of vector bundles, and the proof is completed.
\smartqed\qed\end{proof}

Theorem \ref{thm:endmon1} and Proposition \ref{prop:lotimesf} give
the following explicit description of $\Endhb(E)$.

\begin{theorem}
\label{thm:endmon2} Let $E=\bigoplus_{i,j} L_i\otimes F_{i,j}$ and
$E'=\bigoplus_{i,j} L_i\otimes F'_{i,j}$ be two homogeneous vector
bundles, where $L_i$ are homogeneous line bundles, $F_{ij}$ and
$F'_{ij}$ unipotent homogeneous vector bundles and $L_i\ncong L_j$
if $i\ne j$. Then
\[
\Hom_{hb}(E,E')\cong_{vb}  \bigoplus_i L_i\otimes
\bigl(\oplus_{j,k}\Hom_{hb} (F_{i,j},F'_{i,k}) \bigr).
\]

Moreover, $ \Endhb(E)\cong_{vb}  \bigoplus_i L_i\otimes
\bigl(\oplus_{j,k}\Hom_{hb} (F_{i,j},F_{i,k}) \bigr).$ 
\end{theorem}

\subsection{The algebraic monoid structure of  $\Endhb(E)$}

Let us start with an important consequence of Corollary
\ref{coro:ker}.

\begin{proposition}
\label{prop:kernelend}  The Kernel of $\Endhb(E)$ of a homogeneous
vector bundle $\rho: E\to A$ is given by
\[
\operatorname{Ker}\bigl(\Endhb(E)\bigr)= \Theta\bigl(\Endhb(E)\bigr)=\bigr\{
\theta_a:E\to E\mathrel{:} \theta_a(v)= 0_{\rho(v)+a}\bigr\},
\]
where $\Theta$ is the zero section. Moreover,
$\Ker\bigl(\Endhb(E)\bigr)$ is an algebraic group and isomorphic
to  $A$. 
\end{proposition}

Recall that an endomorphism $f\in \End_A(E)$ is called \emph{nilpotent
of index $n$} if $f^n(v)\in \Theta(E)$ for all $v\in E$ and there exists $v_0\in E$ such that  $f^{n-1}(v_0)\notin
\Theta(E)$. In other words, $f^n=0\in \End_A(E)$ whereas $f^{n-1}\neq 0$. The set $N_A(E)$ of nilpotent endomorphisms is an ideal of $\End_A(E)$, see Section \ref{label:sect22} and \cite{kn:atiyahconn}.

\begin{definition} Let $E\to A$ be a
homogeneous vector bundle. An endomorphism $f\in \Endhb(E)$
 is called {\em pseudo-nilpotent
of index $n$}\/ if $f^n\in \Theta\bigl(\Endhb(E)\bigr) =
\operatorname{Ker}\bigl(\Endhb(E)\bigr)$ whereas $f^{n-1}\notin
\Theta\bigl(\Endhb(E)\bigr)$.  We denote by $\Nhb (E)$ the set of pseudo-nilpotent
endomorphisms. It is clear that $N_A(E)=\Nhb (E)\cap \End_A(E)$.

\end{definition}

If $L$ is a homogeneous line bundle, then
$\Endhb(L)=L$ and  $ \Authb(L)=L\setminus  \Theta(L)$. Hence,
$\Endhb(L)=\Authb(L)\sqcup \Theta(L)$. In particular,
$\Nhb(L)=\Theta(L)=\operatorname{Ker}(L)$. For indecomposable vector bundles of higher rank we have
 an analogue of  Atiyah's results (see
\cite{kn:atiyahconn}).

\begin{theorem}
\label{thm.autnilp} Let $E\to A$ be an indecomposable homogeneous
vector bundle. Then:

 \noindent (1)  The
algebraic monoid $\Endhb(E)$ decomposes as
\[
\Endhb(E)=\Authb(E)\sqcup \Nhb(E).
\]

Moreover,  $\Nhb (E)$  is an ideal of $\Endhb(E)$.

\noindent (2) The set $\Nhb(E)$ of pseudo-nilpotent elements is a
homogeneous vector bundle over $A$ of
$\operatorname{rk}\Nhb(E)=\operatorname{rk}\Endhb(E)- 1$.
Moreover,
\[
\Nhb (E)=\operatorname{Z}^0_{hb}(E)\cdot N_A(E)\cong
\operatorname{Z}^0_{hb}(E)*_{\operatorname{Z}^0_A(E)}N_A(E).
\]
where $\cdot$ denotes  the product in $\Endhb(E)$. The fiber of
$\pi: \Nhb (E)\to A$ is isomorphic to
$N_A(E)$, and $\pi$ is a morphism of algebraic
semigroups.
\end{theorem}

\begin{proof}
From Remark \ref{rem:endwithcenter} we have that $$\Endhb(E)=
\operatorname{Z}_{hb}^0(E)\cdot\End_A(E) \cong
\operatorname{Z}_{hb}^0(E)*_{\operatorname{Z}^0_A(E)}\End_A(E).$$
 Let $\ f\in \End_A(E)$ and $z\in
\operatorname{Z}_{hb}^0(E)$. Since
$\End_A(E)=\bk\operatorname{Id}\oplus N_A(E)$ it
follows that either $f$ is in $\Aut_A(E)$ or in
$N_A(E)$. In the first case $z\cdot f\in \Authb(E)$
and in the second if $f^n=0$, then  $(z\cdot f)^n=z^n\cdot f^n=
\theta_{\pi(z^n)}$, where $\pi:\operatorname{Z}^0_{hb}(E)\to A$ is
the canonical projection. Therefore,
\[
\Endhb(E)=\Authb(E)\sqcup \Nhb(E).
\]

Note that in particular we have proved that
\[
\Nhb (E)=\operatorname{Z}^0_{hb}(E)\cdot N_A(E)\cong
\operatorname{Z}^0_{hb}(E)*_{\operatorname{Z}^0_A(E)} N_A(E).
\]

Since $\Nhb (E)=\Endhb(E)\setminus \Authb(E)$, it follows that
$\Nhb (E)$ is an ideal. In particular,  $\Nhb (E)$ is
$\Authb(E)$-stable, and hence a homogeneous vector bundle.
Finally, the equality
$\operatorname{rk}\Nhb(E)=\operatorname{rk}\Endhb(E)- 1$ follows
again from the fact that $\End_A(E)=\bk\operatorname{Id}\oplus
N_A(E)$ and $N_A(E)\ne 0.$
\smartqed\qed\end{proof}

The next theorem yields information about the
geometric structure of $\Endhb(E)$ when the rank is $\geq 2$.

\begin{theorem}
\label{thm:Endhbext} If $E\to A$ is an indecomposable  vector
bundle of rank $r\geq 2$ obtained by successive extensions of the
homogeneous line bundle $L$, then $\Endhb(E)$ and $\Nhb (E)$ are also
obtained by successive extensions of $L$. Moreover, if $r\geq 2$,
then $\operatorname{rk}\Nhb (E)\geq 1$.
\end{theorem}

\begin{proof}
Let $L'\subset \Endhb(E)$ be a homogeneous, rank-one sub-bundle and let
$f\in L'\cap \End_A(E)=L_0'\setminus \{\theta_0\}$ be a non zero
nilpotent element.  Let $e \in E_0$ be such that $f(e)\neq 0\in
E_0$. Since $\End_{hb}(E)\cong_{vb}
\Aut_{hb}(E)*_{\Aut_A(E)}\End_A(E)$, for every $a\in A$, there
exists $(h_a,t_a)\in \Authb(E)$ such that $L'_a=\Bbbk (h_a\circ
f)$. Hence,  $\varphi: L'\to E$, $\varphi (l)=l(e)$ is an
injective  morphism of homogeneous vector bundles, and since $E$
is obtained by successive extensions of $L$, it follows that
$L'\cong L$. Thus,  $\Endhb(E)$ is also obtained   by successive
extensions of $L$.

From Theorems \ref{thm.autnilp}  and  \ref{thm:Endhbext},
$\Nhb(E)$ is a homogeneous vector bundle, obtained by successive
extensions of $L$. It remains to show that $\operatorname{rk}
\Nhb(E)\geq 1$. But by Proposition \ref{propdimend}, there exists
$0\ne \varphi\in \End_A(E)$ such that $\varphi^2=0$, which is the
desired conclusion.
\smartqed\qed\end{proof}

We are thus led to the following generalization of Miyanishi's
Structure Theorem (see Remark \ref{thm:miyaandmukai}).

\begin{theorem}
\label{thm:sumadirecta} Let $E\cong L\otimes F\to A$ be an
indecomposable homogeneous vector bundle. Then  there exists an
exact sequence  of vector bundles over $A$
\begin{center}
\mbox{ \xymatrix{ 0 \ar@{->}[r] & \Nhb (E)\ar@{->}[r]&
\Endhb(E)\ar@{->}[r]^-{\rho} & \Endhb(L)\cong L\ar@{->}[r]      &
0. } }
\end{center}
 Moreover,
the morphisms in the above sequence are compatible with the
algebraic semigroup structures.  Furthermore, if
$E\not\cong L$, then the sequence is  non-split.
\end{theorem}

\begin{proof}
By Remark
\ref{thm:miyaandmukai},  $L\subset E$ is $\Authb(E)$-stable and $\Endhb(E)$-stable.
It is easy to see that the
restriction $\rho: \Endhb(E)\to \Endhb(L)$ is a morphism of
algebraic monoids. In particular, it is compatible with the underlying
vector bundle structures.

By Theorem \ref{thm.autnilp}, $\Endhb(E)=\Authb(E)\sqcup \Nhb(E)$.
It is clear that if $g\in \Authb(E)$, then $\rho(g)\in
\Authb(L)=L\setminus \Theta(L)$ and if $(f,t_a)\in \Nhb(E)$, there
exists $n\in \mathbb N$ such that $f^n=\theta_{na}$. It follows
that the restriction $f|_{_L}$ belongs to $\Nhb(L)=\Theta(L)$. Therefore,
$\Nhb(E)=\Ker(\rho)$.

Assume now that the exact sequence splits. Then there exists an
immersion of homogeneous vector bundles $\iota:L\hookrightarrow
\Endhb(E)$, such that $\rho\circ\iota=\operatorname{Id}_L$. In
particular, $\iota\bigl(L\setminus \Theta(L)\bigr)\subset
\Authb(E)$. Let $E_0$  be the fiber of $E$ over $0\in A$  and
consider the morphism of vector bundles
\[
\varphi: L\otimes E_0\cong \Endhb(L)\otimes E_0\to E, \quad \quad
\varphi(f\otimes v)= f(v).
\]

Let $e\in E$ be such that $\pi(e)=a$ and $f\in L\setminus
\Theta(L)$ be such that $\alpha(f)=a$. Then
$\varphi\bigl(\iota(f)\otimes \iota(f)^{-1}(e))=e$, and it follows
that $\varphi$ is a surjective morphism of homogeneous vector
bundles of the same rank. Thus, $\varphi$ is an isomorphism. But
$L\otimes E_0$ is decomposable unless $\dim E_0=1$. Therefore,
$E\cong_{vb}L$, and the proof is completed.
\smartqed\qed\end{proof}

\section{Explicit calculations for small rank}
\label{sec:exam}

 The algebra of endomorphisms of successive extensions
of line bundles  over a curve, of small rank,  has been studied
in \cite{kn:lebp1,kn:lebp3, kn:lebp4, kn:lebp5,
  kn:lebp6}. We use a fairly straightforward generalization of such
  results to give an explicit description
of the endomorphisms monoid of indecomposable homogeneous vector
bundles of rank $2$ and $3$ over abelian varieties.

\subsection{Homomorphisms between a homogeneous line bundle and a
 homogeneous vector bundle}

As we saw in Section \ref{sec:endmon}, any homogeneous line bundle
is an algebraic monoid, and is isomorphic to its endomorphisms
monoid. In this section we give a description of $\Hom_{hb}(E,E')$
when one of the homogeneous vector bundles is a line bundle and
the other is an indecomposable  homogeneous vector
bundle.

\begin{proposition}
\label{prop:trivialcase} Let $E=L\otimes F$ be an  indecomposable
homogeneous vector bundle of rank  $\operatorname{rk}E =n\geq 2$,
and $L'$ a homogeneous line bundle. Then,
\begin{enumerate}
\item if $L=L'$,  then
\begin{enumerate} \item  $\Hom_{hb}(L,E)\cong _{hb}\oplus ^rL $ where $r=\dim H^0(A,F)$.
\item  $\Hom_{hb}(E,L)\cong _{hb}\oplus ^sL $ where $s=\dim
H^0(A,F^\vee).$
\end{enumerate}
\item If $L\ne L'$, then $\Hom_{hb}(E,L')\cong _{hb}
\Hom_{hb}(L',E)\cong _{hb}A\times \{0\}.$
\end{enumerate}
\end{proposition}

\begin{proof} From what has been already proven, we have that
\[
\begin{split}
\Hom_{hb}(\cO_A,F) & \cong_{vb}
\Authb(\cO_A)*_{\Aut_A(\cO_A)}\Hom_A(\cO_A,F)\\
& =_{vb} (A\times\bk^*)*_{\bk^*\operatorname{Id}}\Hom_A(\cO_A,F)
\cong_{vb}A\times \Hom_A(\cO_A,F).
\end{split}
\]

It follows that    $\Hom_{hb}(\cO_A,F)$ is a  trivial bundle, with
fiber isomorphic to $\Hom_A(\cO_A,F)=H^0(A,F)$, i.e. $
\Hom_{hb}(\cO_A,F)\cong_{vb}  A\times H^0(A,F)$, which by Proposition \ref{prop:lotimesf} proves (1).

The proof for (2) is similar. In this case we have that $
\Hom_{hb}(F,\cO_A)\cong_{vb}  A\times H^0(A,F^\vee)$.
\smartqed\qed\end{proof}

\subsection{Homomorphisms between  indecomposable homogeneous vector
  bundles of rank $2$}

Let $E$ and $E'$ be two non-isomorphic indecomposable homogeneous
vector bundles of rank $2$. Let
\[
\rho _E: 0\to L \stackrel{j}{\longrightarrow} E
\stackrel{\pi}{\longrightarrow}L \to 0
\]
 and
\[
\rho _{E'}: 0\to L' \stackrel{i_1}{\longrightarrow} E'
\stackrel{p_1}{\longrightarrow} L' \to 0
\]
be the extensions associated with $E$ and $E'$, respectively. By
Proposition \ref{prop:lotimesf}, if $L\ncong _{vb} L'$, then
$\Hom_A(E,E')=0$. If $L\cong_{vb}L'$, then $\Hom_A(E,E')\ne 0$,
since $0\ne \phi= i_1\circ \pi \in \Hom_A(E,E')$.

We are thus led to the following strengthening of Theorem
\ref{thm:Endhbext}.

\begin{proposition} Let $E,E'$ be as above. If $L\cong_{vb}L'$, then $\Hom_{hb}(E,E')\cong_{vb}L$.
\end{proposition}
 \begin{proof}
 It is sufficient to prove that
$\Hom_A(E,E')=\bk\phi$ (see Theorem \ref{thm:Endhbext}).

Let $ 0\neq \varphi\in\Hom_A(E,E')$. Since $E$ and $E'$ are
non-isomorphic, the image $\varphi(E)$ is a line sub-bundle $L_0$
of $E'$. Moreover, $\psi: p_1\circ \varphi : E\to L$ is a non zero
homomorphism of homogeneous vector bundles. If $\varphi(E)=L_0\neq
L$ we get a contradiction, since, by  Proposition
\ref{prop:trivialcase}, $\Hom_{hb}(E,L)\cong_{hb} A\times \{0\}$.
Thus, $L_0\cong_{hb}L$.

If  $ \varphi \ne \lambda \phi$ with $\lambda\in \Bbbk$, $0\ne
\varphi\circ j:L\to L_0$ is an isomorphism, and then $\varphi\circ
(\varphi\circ j)^{-1}:E\to L$ is a spitting. This contradicts our
assumption. Hence, $\varphi=\lambda\phi$, which completes the
proof.
\smartqed\qed\end{proof}

When $E=E'$ we have that $\dim \End_A(E)\leq 2$ (see Theorem
\ref{thm:enhdvb}). To describe $\Endhb(E)$ we consider the
associated exact sequence
\[
0\to L\stackrel{i}{\longrightarrow} E\stackrel{p}{\longrightarrow}
L\to 0
\]
where $L$ is a homogeneous line bundle.

Note that $\operatorname{rk}\Nhb(E)\geq 1,$ since $0\ne \varphi=
i\circ p: E\to E$ satisfies  $\varphi ^2=0$   and
$\operatorname{rk} \Endhb(E)=2$. Hence, $\End_A(E)\cong
\bk\operatorname{Id}\oplus \bk \varphi$
 (see \cite{kn:atiyahconn}). Therefore,

\begin{proposition}
\label{prop:rank2}  Let $E$ be a homogeneous vector bundle of rank 2. Then $\Endhb(E)$ is a commutative algebraic monoid,
and $\End_A(E)\cong \bk[t]/(t^2)$. Moreover,
$\Endhb(E)\cong_{vb}E$.
\end{proposition}
\begin{proof}
The only assertion that  still needs proof here is the last one. For
this, observe that
\[
\End_A(E)\cong_{am} \left\{\left( \begin{smallmatrix}
a & b\\
0 & a
\end{smallmatrix}\right) \mathrel : a,b\in\bk\right\},
\]
with action on the fiber $E_0$ given as follows: consider an
isomorphism  $E_0\cong \bk^2$  such that $(1,0)\in
\Ker(p)_0$. In other words, $(1,0)$ belongs to the fiber $L_0$, of
the line bundle $L\subset E$ which is $\Aut(E)$-stable. Under this
identification, the action $\End_A(E)\times E_0\to E_0$  is given
by $\left(
\begin{smallmatrix}
a & b\\
0 & a
\end{smallmatrix}\right)\cdot (x,y)= (ax+by, ay)$.

On the other hand,  the action of $\Aut_A(E)$ on $\End_A(E)$ is
given by $\left( \begin{smallmatrix}
a & b\\
0 & a
\end{smallmatrix}\right)\cdot \left( \begin{smallmatrix}
x & y\\
0 & x
\end{smallmatrix}\right)=\left( \begin{smallmatrix}
ax & ay+bx\\
0 & ax
\end{smallmatrix}\right)$. Thus, there exists an isomorphism of
$\Aut_A(E)$-modules $\varphi: E_0\to \End_A(E)$, which implies
that the morphism
\[
\begin{array}{rcl}
\psi: \Authb(E)*_{\Aut_A(E)}
E_0 & \to & \Authb(E)*_{\Aut_A(E)}\End_A(E)\\
  \psi\bigl(\bigl[(f,t_a), e_0\bigr]\bigr) &=& \bigl[(f,t_a),
\varphi(e_0)\bigr]
\end{array}
\]
is an isomorphism of  vector bundles, and $E\cong_{vb}
 \Endhb(E)$ as claimed.
\smartqed\qed\end{proof}

\smallskip

\subsection{The endomorphisms monoid of an indecomposable,
  homogeneous,  rank 3  vector
  bundle}

For indecomposable homogeneous vector bundles $E$ of rank $3$,
$\dim \End_A(E)\leq 4$ (see Theorem \ref{thm:enhdvb}). As in
\cite{kn:lebp1,kn:lebp3},
 $\End_A(E)$ is a commutative algebra of dimension
  $2\leq \dim \End_A(E)\leq 3$ and the possibilities are
\[
\End_A(E)=\left\{\begin{matrix}
 \bk[t]/(t^2)  & \text {or}\\
\bk[t]/(t^3)  & \text {or}\\
\bk[r,s]/(r,s)^2 &
\end{matrix}
\right.
\]


The structure of $\End_A(E)$ depends on the extensions associated
with $E$ and their relations. For higher rank there will be more
possibilities for $\End_A(E)$. The study of these cases should help us
find the answer to
 Question \ref{question}.
However, this topic exceeds the scope of this paper.

The remainder of this section is devoted to the study of the case
$\End_A(E)\cong \bk[t]/(t^r)$.

 Assume that $E\to A$ has rank
$r\geq 3$ and $\End_A(E)\cong \bk[t]/(t^r)$.  As in the rank $2$
case, we claim that there exists an isomorphism $\varphi: E_0\to
\End_A(E)$ of $\Aut_A(E)$-modules which induces an isomorphism
\[
\begin{array}{rcl}
\psi: \Authb(E)*_{\Aut_A(E)}
E_0 & \to & \Authb(E)*_{\Aut_A(E)}\End_A(E)\\
  \psi\bigl(\bigl[(f,t_a), e_0\bigr]\bigr) &=& \bigl[(f,t_a),
\varphi(e_0)\bigr]
\end{array}
\]
 of  vector bundles and hence, $E\cong_{vb}
 \Endhb(E)$. We give the proof of the claim only for the case $r=3$.
The proof of the general case is similar, and is
left to the reader.

\begin{proposition}
\label{prop:rank3} If $\operatorname{rk}(E)=3$ and $\End_A(E)\cong  \bk[t]/(t^3)$,
then $\Endhb(E)$ is a commutative algebraic monoid,  and
$\Endhb(E)\cong_{vb}E$.
\end{proposition}

\begin{proof} It suffices to prove that the representations
  $\Aut_A(E)\times \End_A(E)\to \End_A(E)$ and $\Aut_A(E)\times E_0\to
  E_0$ are isomorphic. In this case,
 \[
\End_A(E)\cong_{am} \left\{\left( \begin{smallmatrix}
a & b & c\\
0 & a & b\\
0& 0& a
\end{smallmatrix}\right) \mathrel a,b\in\bk\right\},
\]
and the action over the fiber $E_0$ is given as follows: consider
an isomorphism $E_0\cong \bk^3$, such that $(1,0,0)\in (E_1)_0$,
where $L=E_1\subset E_2\subset E$ is a $\Aut(E)$-stable
filtration. Under this identification the action  $\End_A(E)\times
E_0\to E_0$  is given by $\left( \begin{smallmatrix}
a & b & c\\
0 & a & b\\
0& 0& a
\end{smallmatrix}\right)
\cdot (x,y,z)= (ax+by+cz, ay+bz,az)$.

On the other hand,  the action of $\Aut_A(E)$ on $\End_A(E)$ is
given by

$\left( \begin{smallmatrix}
a & b & c\\
0 & a & b\\
0& 0& a
\end{smallmatrix}\right) \cdot
 \left( \begin{smallmatrix}
x & y & z\\
0 & x & y\\
0 & 0 & x
\end{smallmatrix}\right)=\left( \begin{smallmatrix}
ax & bx+ay & ax+by+cz\\
0 & ax & bz+ay\\
0&0& ax
\end{smallmatrix}\right)$.

Therefore, there exists an isomorphism $\varphi: E_0\to \End_A(E)$
of $\Aut_A(E)$-modules  and hence the homomorphism
\[
\begin{array}{rcl}
\psi: \Authb(E)*_{\Aut_A(E)}
E_0 & \to & \Authb(E)*_{\Aut_A(E)}\End_A(E)\\
  \psi\bigl(\bigl[(f,t_a), e_0\bigr]\bigr) &=& \bigl[(f,t_a),
\varphi(e_0)\bigr]
\end{array}
\]
is an isomorphism of  vector bundles and $E\cong_{vb}
 \Endhb(E)$ as claimed.
\smartqed\qed\end{proof}

\begin{acknowledgement}  The first author is a
member of the research group VBAC (Vector Bundles on Algebraic
Curves). The second author wishes to thank CIMAT (Guanajuato,
Mexico) where the paper was written, for the invitation and
hospitality, and to ANII and CSIC (Uruguay) for partial support. Both
authors acknowledge the support of CONACYT. We also thank Lex Renner
for many helpful suggestions.
\end{acknowledgement}

\end{document}